\documentclass{amsart}
\usepackage[margin=3.5cm]{geometry}
\usepackage{xcolor}

\usepackage{amsfonts, amsthm, amssymb,verbatim,amscd,amsmath, blindtext}

\usepackage{multirow}
\usepackage{graphicx}
\usepackage{enumitem,yhmath}
\usepackage{amssymb}
\usepackage{ytableau}
\usepackage{tikz}
\usepackage{tikz-cd}
\usepackage{float, pifont}
\usepackage{mathtools}
\usepackage[pagebackref,colorlinks=true,citecolor=blue,linkcolor=black]{hyperref}

\newtheorem{theorem}{Theorem}[section]
\newtheorem{lemma}[theorem]{Lemma}

\newtheorem{proposition}[theorem]{Proposition}

\newtheorem*{maintheorem}{Main Theorem}

\theoremstyle{remark}
\newtheorem{remark}[theorem]{Remark}

\theoremstyle{definition}
\newtheorem{definition}[theorem]{Definition}

\newtheorem{algorithm}[theorem]{Algorithm}

\makeatletter

\def\env@sqcases{%
  \let\@ifnextchar\new@ifnextchar
  \left\lbrack
  \def\arraystretch{1.2}%
  \array{@{}l@{\quad}l@{}}%
}
\makeatother

\DeclareMathOperator{\lcm}{lcm}

\DeclareMathOperator{\sbridge}{sb}
\DeclareMathOperator{\Taylor}{Taylor}

\DeclareMathOperator{\mingens}{Mingens}

\begin{document}
\title[Monomial ideals with five generators]{\textbf{Minimal cellular resolutions of monomial ideals with five generators and their Artinian reductions}}
\author{Trung Chau}
\address{Chennai Mathematical Institute, Siruseri, Tamil Nadu, India}
\email{chauchitrung1996@gmail.com}

\begin{abstract}
    We prove that monomial ideals with at most five generators and their Artinian reductions have minimal generalized Barile-Macchia resolutions. As a corollary, these ideals have minimal cellular resolutions, extending a result by Faridi, D.G, Ghorbanic, and Pour. This corollary is independently obtained by Montaner, Garc\'{i}a, and Mafi, using a different class of cellular resolutions called pruned resolutions.
\end{abstract}

\maketitle

\section{Introduction}

It is a central problem over the last few decades to study minimal free resolutions of monomial ideals over a polynomial ring over a field. It is the work of Diana Taylor \cite{Tay66} that pioneered the explicit construction of free resolutions of a monomial ideal with her now well-known Taylor resolutions. These, however, are generally accepted to be highly non-minimal, especially when the monomial ideal has many generators. In attempts to obtain free resolutions closer to the minimal one, 
subcomplexes of the Taylor resolutions which are also resolutions are considered \cite{BPS98, Ly88, Nov00, Mer09, FHHM24,CHM24-2}. Alternatively, researchers have generalized this approach to consider free resolutions induced by CW-complexes and successfully construct alternatives to Taylor resolutions \cite{BS98, OW16, pivot} or even minimal resolutions for special ideals \cite{CN08, CN08-2, OY2015}. It is noteworthy that no all minimal resolutions are induced by CW-complexes, as shown by Velasco \cite{Vel08}.

An approach we will focus on is cellular resolutions from discrete Morse theory, which we shall call \emph{Morse resolutions}. Batzies and Welker \cite{BW02} were the first to apply this theory to the study of monomial resolutions, building up from the work of Chari \cite{Cha00} and Forman \cite{Fo94}. It has been a popular and fertile approach for researchers, especially in the last decade \cite{ALCO_2024__7_1_77_0,CEFMMSS22, BM20, AFG2020, Montaner24, faridi24, JW09}. The idea is that the Taylor resolution of a given monomial ideal $I$ is induced from a simplex, and discrete Morse theory allows users to find a CW-complex which is homotopy equivalent to the simplex, with less cells. This CW-complex then induces a free resolution that is closer to the minimal resolution that the Taylor resolution itself. It is, however, not always minimal.

A subclass of Morse resolutions that will serve as the main tool of this paper is \textbf{generalized Barile-Macchia resolutions}. These resolutions were defined in a work of the author and Kara \cite{CK24}, and since has been proved to be minimal in a variety of cases \cite{CHM24-1, CKW24}. In this paper, we extend existing results with a new class of monomial ideals.

\begin{maintheorem}\makeatletter\def\@currentlabel{Main Theorem}\makeatother\label{main}
    Let $J$ be a monomial ideal with at most five generators in $S=\Bbbk[x_1,\dots, x_N]$ where $\Bbbk$ denotes a field. Set \[
    I\coloneqq J+ (x_{1}^{n_1},\dots, x_{N}^{n_N}),
    \]
    where $n_1,\dots, n_N$ are positive integers. 
    Then $I$ and $J$ have a minimal generalized Barile-Macchia resolution. In particular, $I$ and $J$ have minimal cellular resolutions.
\end{maintheorem}

We note that the ideal $I$ defined as above is called an \emph{Artinian reduction} of $J$, as $I$ is readily Artinian. This result extends the previously known one for monomial ideals with four generators or less and their Artinian reductions by Faridi, D.G, Ghorbanic, and Pour \cite[Main Theorem]{faridi24}. Moreover, monomial ideals with five generators and their Artinian reductions have been recently proved to have minimal pruned free resolutions, a different subclass of Morse resolutions, independently by Montaner, Garc\'{i}a, and Mafi \cite[Corollary 6.17]{Montaner24}. 

It is noteworthy that this result cannot be extended to monomial ideals with six generators, since Morse resolutions, and thus generalized Barile-Macchia resolutions in particular, are independent of the characteristic of $\Bbbk$, while there exists a monomial ideal with six generators whose minimal resolutions do depend on it (see \cite{KTY09}, or \cite[Example 6.4]{CK24} for this note). 

This paper is structured as follows. Section 2 provides some background on Morse and generalized Barile-Macchia resolutions. Section 3 is dedicated to the proof of the \ref{main}, with the two large cases proven separately in Propositions~\ref{prop:4} and \ref{prop:5}.

\section*{Acknowledgements}

The author would like to thank Professor Montaner for sending them an early draft of his paper with Garc\'{i}a and Mafi \cite{Montaner24}, and also for his advice and encouragement on the completion of this project. The author would like to thank Professor Faridi for the discussions regarding the paper \cite{faridi24}. The author acknowledges the support of a grant from Infosys foundation.

\section{Preliminaries}

In this subsection, we give a brief overview of how to apply discrete Morse theory to the Taylor resolution of any monomial ideal to construct its generalized Barile-Macchia resolutions. Throughout the section, $S = \Bbbk[x_1, \dots, x_N]$ denotes a polynomial ring over a field $\Bbbk$ and $I$ is a monomial ideal in $S$.

We quickly recall the definition of free resolutions. A \emph{free resolution} of $S/I$ is a complex of free $S$-modules
\[
\mathcal{F}\colon 0\to F_r \xrightarrow{\partial} \cdots \xrightarrow{\partial} F_1 \xrightarrow{\partial} F_0 \to 0,
\]
such that $H_i(\mathcal{F})$ is isomorphic to $S/I$ when $i=0$, and $0$ otherwise. Moreover, $\mathcal{F}$ is called \emph{minimal} if $\partial(F_i)\subseteq (x_1,\dots, x_N)F_{i-1}$ for any integer $i$.

Let $\mingens(I)$ be the set of minimal monomial generators of $I$, and let $\mathcal{P}(\mingens(I))$ be its power set. Let $\lcm(I)$ denote the \emph{lcm lattice} of $I$ \cite{Gasharov1999TheLI}, i.e., the lattice of $\{\lcm(\sigma) \colon \sigma\subseteq \mingens(I)\}$, ordered by divisibility. We define the following map:
\begin{align*} 
\lcm \colon \mathcal{P}(\mingens(I)) &\longrightarrow \lcm(I) \\
\sigma &\mapsto \lcm(\sigma).
\end{align*}

Set $\mingens(I)=\{m_1,\dots, m_q\}$.  For any two sets $\sigma,\tau \in \mathcal{P}(\mingens(I))$ such that $\tau\subseteq \sigma$ and $|\tau|=|\sigma|-1$, we can assume that $\sigma=\{m_{i_1},\dots, m_{i_r}\}$ for some indices $1\leq i_1<i_2<\cdots < i_r\leq q$, and $\{m_{i_j}\}\coloneqq \sigma \setminus \tau$. Then we set 
\[
[\sigma\colon \tau] \coloneqq (-1)^{j+1}.
\]

Let $\Taylor(I)$ denote the complex of free $S$-modules
\[
0\to T_q \xrightarrow{\partial}  \cdots \xrightarrow{\partial} T_1 \xrightarrow{\partial} T_0 \to 0
\]
where 
\[
T_i\coloneqq \bigoplus_{\substack{\sigma \in \mathcal{P}(\mingens(I))\\ |\sigma|=i }} Se_\sigma
\]
and 
\[
\partial(e_\sigma) = \sum_{\substack{\tau \in \mathcal{P}(\mingens(I))\\ \tau \subseteq |\sigma, |\tau| = |\sigma|-1 }}  [\sigma\colon \tau] \frac{\lcm(\sigma)}{\lcm(\tau)} e_\tau
\]
for any $\sigma\in \mathcal{P}(\mingens(I))$. The complex $\Taylor(I)$ is a free resolution of $S/I$, and more commonly known as the \emph{Taylor resolution} of $S/I$ \cite{Tay66}.

Discrete Morse theory is designed to cut down cells of a CW-complex to obtain one that is homotopy equivalent to the original. The Taylor resolution of $S/I$ is induced by the full $q$-simplex with vertices labeled with generators of $I$. Since Batzies and Welker \cite{BW02}, it has been a commonly used technique to apply discrete Morse theory to the Taylor resolutions directly. We recall the following description from \cite{BW02}. 

We associate $I$ with a directed graph $G_I=(V,E)$, whose vertex and edge sets are
\[
V=\{\sigma \mid \sigma \subseteq \mingens(I)\}
\]
and
\[
E=\{\sigma\to \tau  \mid \tau\subset \sigma \text{ and }|\tau|=|\sigma|-1 \}.
\]

The main objects of discrete Morse theory are defined as follows. 

\begin{definition} \label{def.Morse}
A collection of edges $A\subseteq E$ in $G_I$ is called an \emph{homogeneous acyclic matching} if the following conditions hold:
    \begin{enumerate}
        \item Each vertex of $G_I$ appears in at most one edge of $A$.
        \item For each directed edge $\sigma\to \tau$ in $A$, we have $\lcm(\sigma)=\lcm(\tau)$.
        \item The directed graph $G_I^A$ --- which is $G_I$ with directed edges in $A$ being reversed --- is acyclic, i.e., $G_I^A$ does not have any directed cycle.
    \end{enumerate}
\end{definition}

Homogeneous acyclic matchings induce Morse resolutions of $S/I$ (\cite[Proposition 2.2]{BW02}). To completely describe the free modules and the differentials of these Morse resolutions, we recall the following terminology.

Let $A \subseteq E$ be a homogeneous acyclic matching  in $G_I$. For a directed edge $(\sigma \to \tau) \in E(G^A_I)$, we set 
\[
m(\sigma,\tau)\coloneqq \begin{cases}
    -[\tau:\sigma] & \text{ if } (\tau\to \sigma) \in A,\\
     ~~[\sigma:\tau] & \text{ otherwise}.
\end{cases}
\]
A \textit{gradient path} $\mathcal{P}$  from $\sigma_1$ to $\sigma_t$ is a directed path $\mathcal{P}\colon \sigma_1 \to \sigma_2\to \cdots \to \sigma_t$ in  $G^A_I$. Set
\[ m(\mathcal{P})=m(\sigma_1,\sigma_2)\cdots m(\sigma_{t-1}, \sigma_t).\]
It is important to note that $m(\mathcal{P})=\pm 1$ for any gradient path $\mathcal{P}$.

The subsets of $\mingens(I)$ that are not in any edge of $A$ are called \emph{$A$-critical}. When there is no confusion, we will simply use the term \emph{critical}. The main result of discrete Morse theory essentially says that the critical subsets of $\mingens(I)$ form a free resolution for $S/I$. 

\begin{theorem}[{\cite[Propositions 2.2 and 3.1]{BW02}}] \label{thm:morse-resolutions}
    Let $A$ be a homogeneous acyclic matching in $G_I$. Then, $A$ induces a free resolution $\mathcal{F}_A$ of $S/I$, which we call the \emph{Morse resolution} with respect to $A$. For each integer $i \ge 0$, a basis of $(\mathcal{F}_A)_i$ can be identified with the collection of critical subsets of $\mingens(I)$ with exactly $i$ elements, and the differentials are defined to be 
    \[ \partial_r^A(\sigma) =\sum_{\substack{\sigma'\subseteq \sigma,\\|\sigma'|=r-1}} [\sigma:\sigma'] \sum_{\substack{\sigma'' \text{ critical,}\\ |\sigma''|=r-1 }} \sum_{\substack{\mathcal{P} \text{ gradient path}\\ \text{from } \sigma' \text{ to }\sigma''}} m(\mathcal{P}) \frac{\lcm(\sigma)}{\lcm(\sigma'')} \sigma''. \]
\end{theorem}

We recall a method to produce homogeneous acyclic matchings for any monomial ideal: the Barile-Macchia algorithm \cite{CK24}. 

\begin{definition} \label{defi:types}
Fix a total ordering $(\succ)$ on $\mingens(I)$.
\begin{enumerate} 
    \item Given $\sigma\subseteq \mingens(I)$ and $m\in \mingens(I)$  such that $\lcm(\sigma \cup \{m\})=\lcm(\sigma\setminus \{m\})$, we say that $m$ is a \emph{bridge} of $\sigma$ if $m\in \sigma$.
    \item If $m\succ m'$ where $m,m'\in \mingens(I)$, we say that $m$ \emph{dominates} $m'$.
    \item The \emph{smallest bridge function} is defined to be
    \[
    \sbridge_{\succ}\colon  \mathcal{P}(\mingens(I))\to \mingens(I) \sqcup \{\emptyset\}
    \]
    where $\sbridge_{\succ}(\sigma)$ is the smallest bridge of $\sigma$ (with respect to $(\succ)$) if $\sigma$ has a bridge and $\emptyset$ otherwise.
\end{enumerate}
\end{definition}

\begin{algorithm}\label{algorithm1}
    {\sf Let $A=\emptyset$. Input: a total ordering $(\succ)$ on $\mingens(I)$ and a given $\Omega\subseteq \{\text{all  subsets of } \mingens(I) \}.$
    \begin{enumerate}[label=(\arabic*)]
        \item Pick a subset $\sigma$ of maximal cardinality in $\Omega$. 
        \item  Set
        \[ \Omega \coloneqq \Omega \setminus \{\sigma, \sigma \setminus \{\sbridge_{\succ} (\sigma)\}\}. \]
        If  $\sbridge_{\succ} (\sigma)\neq \emptyset$, add the directed edge $\sigma \to (\sigma \setminus \{\sbridge_{\succ} (\sigma)\})$ to $A$. \newline
        If $\Omega\neq \emptyset$, return to step (1).
        \item Whenever there exist distinct directed edges $\sigma \to (\sigma \setminus \{\sbridge_{\succ} (\sigma)\})$ and $\sigma'\to ( \sigma' \setminus \{\sbridge_{\succ} (\sigma')\})$ in $A$ such that 
            $$\sigma \setminus \{\sbridge(\sigma)\} = \sigma' \setminus \{\sbridge(\sigma')\},$$
            then 
            \begin{itemize}
                \item if $ \sbridge_{\succ} (\sigma') \succ \sbridge_{\succ} (\sigma)$, remove $\sigma'\to ( \sigma' \setminus \{\sbridge_{\succ} (\sigma')\})$ from $A$,
                \item else remove $\sigma\to ( \sigma \setminus \{\sbridge_{\succ} (\sigma)\})$ from $A$.
            \end{itemize}
        
        \item Return $A$.
    \end{enumerate}}
\end{algorithm}

The algorithm returns a homogeneous acyclic matching \cite{CK24}. By applying this algorithm to each component of the partition of $\mathcal{P}(\mingens(I))$ based on their monomial label, we obtain the following result.

\begin{theorem}[{\cite[Theorem 5.18]{CK24}}]
    \label{thm:generalized-BM} 
     Let $(\succ_p)_{p\in \lcm(I)}$ a sequence of total orderings of $\mingens(I)$. For each $p\in \lcm(I)$, let $A_{\succ_p}$ be the homogeneous acyclic matching obtained by applying Algorithm~\ref{algorithm1} to the set $\lcm^{-1}(p)$ using the total ordering $(\succ_p)$. Then $A=\bigcup_{p\in \lcm(I)}A_{\succ_p}$ is a homogeneous acyclic matching. Hence, it induces a free resolution of $S/I$, called the \emph{generalized Barile-Macchia resolution} induced by $(\succ_p)_{p\in \lcm(I)}$.
\end{theorem}

In \cite{CK24} the generalized Barile-Macchia resolution is defined in a more general setting, replacing $\lcm(I)$ with any poset that is ``compatible" with it. For the context of this paper, this version will be sufficient. We now recall some terminology from \cite{CK24} to help study this resolution.

\begin{definition} \quad \quad \label{defi:types-2}
    Let $\sigma\in \mathcal{P}(\mingens(I))$ and $m\in \mingens(I)$. Set $p\coloneqq \lcm(\sigma)$, and let $(\succ_p)$ denote a total ordering on $\mingens(I)$.
\begin{enumerate}
    \item The monomial $m$ is called a \emph{gap} of $\sigma$ if $m\notin \sigma$.
    \item The monomial $m\in \mingens(I)$ is called a \emph{true gap} of $\sigma$ (w.r.t. $(\succ_p)$) if 
        \begin{enumerate}
            \item[(a)]  it is a gap of $\sigma$, and 
            \item[(b)]  the set $\sigma \cup \{m\}$ has no new bridges dominated by $m$. In other words, if $m'$ is a bridge of $\sigma \cup \{m\}$ and $m\succ_p m'$, then $m'$ is a bridge of $\sigma$.
        \end{enumerate}
    Equivalently, $m$ is not a true gap of $\sigma$ (w.r.t. $(\succ_p)$) either if $m$ is not a gap of $\sigma$ or if there exists $m'\prec_p m$ such that $m'$ is a bridge of $\sigma \cup \{m\}$ but not one of $\sigma$. In the latter case, we call $m'$ a \emph{non-true-gap witness} of $m$ (w.r.t. $(\succ_p)$) in $\sigma$.
    \item The set $\sigma$ is called \emph{potentially-type-2} (w.r.t. $(\succ_p)$) if it has a bridge not dominating any of its true gaps, and \emph{type-1} (w.r.t. $(\succ_p)$) if it has a true gap not dominating any of its bridges. Moreover, $\sigma$ is called \emph{type-2} (w.r.t. $(\succ_p)$) if it is potentially-type-2 and whenever there exists another potentially-type-2 $\sigma'$ such that 
    \begin{equation*} 
    	\sigma' \setminus \{\sbridge_{\succ_p} (\sigma')\}=\sigma \setminus \{\sbridge_{\succ_p}(\sigma)\},
    \end{equation*}
    we have $\sbridge_{\succ_p}(\sigma')\succ \sbridge_{\succ_p}(\sigma)$.
\end{enumerate}
\end{definition}

\begin{remark}
    Given a subset $\sigma\in \mathcal{P}(\mingens(I))$, the concepts of true gaps of $\sigma$, and whether $\sigma$ is type-1, type-2, or potentially-type-2, all depend on a given total ordering on $\mingens(I)$, while the concepts of bridges and gaps of $\sigma$ do not. When only one total ordering $(\succ)$ on $\mingens(I)$ is given, we will implicitly use it when we refer to true gaps of $\sigma$, or when we determine which type $\sigma$ is. When given a family of total orderings $(\succ_p)_{p\in \lcm(I)}$ on $\mingens(I)$, we will use $(\succ_{\lcm(\sigma)})$ to do so.
\end{remark}

For the rest of this subsection, let $(\succ_p)_{p\in \lcm(I)}$ be a family of total orderings  on $\mingens(I)$, and let $A$ denote the homogeneous acyclic matching induced from them in Theorem~\ref{thm:generalized-BM}. 
The idea of the above definitions is that every type-1 set is corresponds to a unique type-2 set, and vice versa. Moreover, we have
\[
A= \{ (\sigma \to \sigma\setminus \{\sbridge_{\succ_{\lcm(\sigma)}}(\sigma)\}) \colon \sigma \text{ is type-2} \}
\]
from \cite[Definition 2.16]{CK24}, while Definition~\ref{defi:types-2} (3) is from \cite[Theorem 2.24]{CK24}. The potentially-type-2 sets are the sources of the directed edges in $A$ after step (2) of Algorithm~\ref{algorithm1}, but removed in its step (3). The following result then follows immediately.

\begin{lemma}\label{lem:critical-two-ways}
    Let $\sigma\subseteq \mingens(I)$ be a critical set. Then either of the following holds:
    \begin{enumerate}
        \item $\sigma$ is neither type-1 nor potentially-type-2. In particular, $\sigma$ has no bridge.
        \item $\sigma$ is potentially-type-2, but not type-2.
    \end{enumerate}
\end{lemma}

While it is generally time-consuming to determine whether a given set is critical, or of which type, in one special case we can tell right away. 

\begin{lemma}\label{lem:losing-the-end}
    Given a $\sigma\in \mathcal{P} (\mingens(I))$ and set $p=\lcm(\sigma)$. Assume that $m\in \mingens(I)$ is the smallest monomial with respect to $(\succ_p)$. If $m$ divides $\lcm(\sigma\setminus \{m\})$, then $(\sigma\cup\{m\} \to \sigma\setminus \{m\})$ is an edge in the homogenous acyclic matching $A$. In particular, $\sigma$ is not critical.
\end{lemma}
\begin{proof}
    Since $m$ is the smallest monomial with respect to $(\succ_p)$, no monomial in $\mingens(I)$ dominates it. Since $m$ is also a bridge of $\sigma\cup \{m\}$, this set is potentially-type-2, and hence type-2 since no monommial in $\mingens(I)$ dominates $m$. This concludes the proof.
\end{proof}

\section{Proof of the \ref{main}}

We dedicate this section to proving the \ref{main}. Set $I=J+(x_{1}^{n_1},\dots, x_{N}^{n_N})$ where $J$ is generated by at most five generators, and $n_i\geq 1$ for any $i\in [N]$. Without loss of generality, we assume that $\mingens(I)=\mingens(J)\cup \{x_{1}^{n_1},\dots, x_{N}^{n_N}\}$. We will consider $\lcm$ as a function on $\mathcal{P}(\mingens(I))$.

We observe that there are not many choices for a bridge or a gap of a subset of $\mingens(I)$.

\begin{lemma}\label{lem:bridge-gap-J}
    If a monomial $m$ is a bridge or a gap of $\sigma\in \mathcal{P}(\mingens(I))$, then $m\in \mingens(J)$. 
\end{lemma}
\begin{proof}
    We will prove this statement with contraposition. Suppose that $m=x_i^{n_i}$ for some $i\in [N]$, and we will show that $m$ is neither a bridge nor a gap of $\sigma$.
    
    Since $\mingens(I)=\mingens(J)\cup \{x_{1}^{n_1},\dots, x_{N}^{n_N}\}$, the only monomial in $\mingens(I)$ that is divisible by $m$ is $m$ itself. Hence $\lcm(\sigma \cup \{m\}) \neq \lcm(\sigma \setminus \{m\})$ as the former is divisible by $m$, while the latter is not. Therefore, by definition, $m$ is neither a bridge nor a gap of $\sigma$, as desired. 
\end{proof}

\begin{remark}\label{rem:only-J}
    It is important to note that the total ordering $(\succ)$ on $\mingens(I)$ in Algorithm~\ref{algorithm1} is used to determine the smallest bridge of a subset of $\mingens(I)$, and thus the monomials that cannot be bridges or gaps of any subset of $\mingens(I)$ do not matter regarding their positions in $(\succ)$. Therefore, by the above lemma, it suffices to restrict $(\succ)$ to $\mingens(J)$.
\end{remark}

Due to the structure of the monomial ideal $I$, the set $\lcm^{-1}(p)$ for any monomial $p\in \lcm(I)$ has some restrictions.  Set $M_p \coloneqq \{m\in \mingens(J) \colon m \mid p \}$.

\begin{lemma}\label{lem:lcm-structure}
    Given a monomial $p\in \lcm(I)$. Then there exist indices $i_{1},\dots, i_{t}$ for some integer $t$ such that for any $\sigma\in \lcm^{-1}(p)$, we have
    \[
    M\subseteq \sigma \subseteq M \cup M_p,
    \]
    where $M\coloneqq \{x_{i_{1}}^{n_{i_1}}, \dots, x_{i_{t}}^{n_{i_t}}\}$.
\end{lemma}
\begin{proof}
    Consider $\sigma\in \lcm^{-1}(p)$. Let $i_{1},\dots, i_{t}$ be all the indices $l$ such that $x_{l}^{n_l}\mid \lcm(\sigma)$. Since elements of $\{x_{1}^{n_1},\dots, x_{N}^{n_N}\}$ are powers of a variable and minimal monomial generators of $I$, they belong to $\sigma$ if and only if they divide $\lcm(\sigma)$. In other words, we have
    \[
    M\subseteq \sigma \subseteq M\cup \mingens(J).
    \]
    Any monomial in $\mingens(J)\setminus M_p$ does not divide $p$, and thus they are not in $\sigma$. Thus $\sigma \subseteq M\cup M_p$, as desired.
\end{proof}

By Theorem~\ref{thm:morse-resolutions}, for a resolution induced by a homogeneous acyclic matching to be non-minimal, it is necessary that a gradient path of certain properties to exist. We study the structure of such a gradient path in the next results.

\begin{definition}
    Let $A$ be a homogeneous acyclic matching of $I$, and $p\in \lcm(I)$ a monomial. Then a \emph{bad gradient path of type $p$} in $G_I^A$ is a gradient path of the form
    \[
    \sigma_1 \to \tau_1 \to \sigma_2\to \tau_2 \to \cdots \to \sigma_{l} \to \tau_l
    \]
    for some $l\geq 1$ such that
    \begin{enumerate}
        \item $\sigma_1$ and $\tau_l$ are critical with $|\sigma_1|=|\tau_l|+1$;
        \item $\lcm(\sigma_i)=\lcm(\tau_j)=p$ for any $i,j\in [l]$.
    \end{enumerate}
\end{definition}

\begin{lemma}\label{lem:gradient-path-shape}
    Given two $A$-critical sets $\sigma$ and $\sigma'$ with $|\sigma|=|\sigma'|+1$, where $A$ is a homogeneous acyclic matching for a monomial ideal $I$. Then a gradient path from $\sigma$ to $\sigma'$ is of the form
    \[
    \sigma_1 \to \tau_1 \to \sigma_2 \to \tau_2 \to \cdots \to \sigma_l\to \tau_l
    \]
    for some $l\geq 1$, where $\sigma_1=\sigma$, $\tau_l=\sigma'$, $(\sigma_{i+1}\to \tau_i) \in A$, and $(\sigma_i\to \tau_i) \in E(G_I)\setminus A$,  for any $i\in [l]$.
\end{lemma}
\begin{proof}
    Directed edges in $G_I^A$ are either in $E(G_I)\setminus A$, or of the form $(\eta\to\mu)$ where $(\mu\to \eta) \in A$. We note that directed edges of the former type go from a set to another with one less cardinality, and those of the latter type go from a set to another with one more cardinality. Since $|\sigma|=|\sigma'|+1$, a gradient path from $\sigma$ to $\sigma'$ must have $l$ directed edges in $E(G_I)\setminus A$, and $l-1$ directed edges whose reverses are in $A$, for some positive integer $l$. 
    
    Since $A$ is a homogeneous acyclic matching, the directed edges in $A$, pairwise, do not share any common vertex, and thus cannot be consecutive in a gradient path. Finally, since $\sigma$ and $\sigma'$ are critical, the first and last edges of a gradient path from $\sigma$ to $\sigma'$ cannot be inverses of directed edges in $A$. Therefore, the $l$ directed edges in $E(G_I)\setminus A$, and $l-1$ directed edges whose reverses are in $A$ in a gradient path from $\sigma$ to $\sigma'$ must alternate, starting with a directed edge in $E(G_I)\setminus A$, as desired.
\end{proof}

\begin{lemma}\label{lem:bad-gradient-path-exists}
    Let $A$ be a homogeneous acyclic matching of $I$. Assume that the Morse resolution induced by $A$ is not minimal, then a bad gradient path (of some type) in $G_I^A$ exists.
\end{lemma}

\begin{proof}
    By Theorem~\ref{thm:morse-resolutions}, if the Morse resolution induced by $A$ is not minimal, then there exist two critical sets $\sigma$ and $\sigma'$ such that $|\sigma|=|\sigma'|+1$, $\lcm(\sigma)=\lcm(\sigma')$, and there exists a gradient path from $\sigma$ to $\sigma'$.
    By Lemma~\ref{lem:gradient-path-shape}, this gradient path is of the form
    \[
    \sigma_1 \to \tau_1 \to \sigma_2 \to \tau_2 \to \cdots \to \sigma_l\to \tau_l
    \]
    for some $l\geq 1$, where $\sigma_1=\sigma$, $\tau_l=\sigma'$, $(\sigma_{i+1}\to \tau_i)$ is an edge $A$, and $(\sigma_i\to \tau_i)$ is an edge of $G_I$ but not in $A$, for any $i$. By definition, we have 
    \[
    \lcm(\tau_l) \mid \lcm(\sigma_{l}) = \lcm(\tau_{l-1}) \mid \cdots \mid \lcm(\sigma_2)=\lcm(\tau_1) \mid \lcm(\sigma_1). 
    \]
    Couple this with the fact that $\lcm(\sigma_1)=\lcm(\tau_l)$, we conclude that all the sets in this gradient path have the same lcm. Thus, this a bad gradient path, as desired.
\end{proof}

We can say more about a bad gradient path if the homogeneous acyclic matching induces a generalized Barile-Macchia resolution.

\begin{lemma}\label{lem:gradient-path-4-monomials}
    Let $(\succ_p)_{p\in \lcm(I)}$ be a family of total orderings on $\mingens(I)$ and $A$ the homogeneous acyclic matching induced from them.  Assume that there exists a bad gradient path of the form
    \[
    \sigma_1 \to \sigma_2 \to \cdots  \to \sigma_l
    \]
    for some $l\geq 2$.
    Then there exist 4 monomials $m_1\succ m_2\succ m_3\succ m_4$ such that
    \begin{enumerate}
        \item $m_1$ is a bridge of $\sigma_1$ and $\sigma_2=\sigma\setminus \{m_1\}$;
        \item $m_2=\sbridge(\sigma_1)$;
        \item $m_3$ is a gap, but not a true gap of $\sigma_1$;
        \item $m_4$ is a non-true-gap witness of $m_3$ in $\sigma$.
    \end{enumerate}
    In particular, $m_1,m_2,m_3,m_4\in \mingens(J)$.
\end{lemma}

\begin{proof}
    We have $\sigma_2=\sigma_1 \setminus \{m_1\}$ and $\lcm(\sigma_2)=\lcm(\sigma_1)$ for some monomial $m_1\in \mingens(I)$. By definition, $m_1$ is a bridge of $\sigma_1$. Thus (1) follows. Since $\sigma_1$ has one bridge, $m_2\coloneqq \sbridge(\sigma_1)$ exists, which implies $(2)$.

    By Lemma~\ref{lem:critical-two-ways}, $\sigma$ is potentially-type-2, but not type-2. In other words, there exists $\tau\in \mathcal{P}(\mingens(I))$ such that $\tau \neq \sigma_1$, $\tau\setminus \{\sbridge(\tau)\}= \sigma_1\setminus \{\sbridge(\sigma_1)\}$, and $\sbridge(\tau)\prec \sbridge(\sigma_1) = m_2$. Set $m_3\coloneqq \sbridge(\tau)$. Since $\sigma_1$ is potentially-type-2, the monomial $m_3$ is not a true gap of $\sigma_1$, which implies (3). Moreover, this also implies that there exists a monomial $m_4$ that serves as a non-true-gap witness of $m_3$ in $\sigma_1$, which is (4). The existence of these four monomials then follows, while the last statement follows from Lemma~\ref{lem:lcm-structure}.
\end{proof}

This lemma shows that, in particular, a bad gradient path does not exist if $J$ has at most 3 generators, no matter which generalized Barile-Macchia resolution is used. The key of the proof of the \ref{main} is, if $J$ has 4 or five generators, one can construct a system of total orderings such that using the homogeneous acyclic matching induced by them in Theorem~\ref{thm:generalized-BM}, a bad gradient path does not exist. We will deal with the case of $4$ and $5$ generators separately.

\begin{proposition}\label{prop:4}
    Given a monomial $p\in \lcm(I)$. Assume that $|M_p|=4$. Let $A$ be the homogeneous acyclic matching resulted from applying Algorithm~\ref{algorithm1} to $\lcm^{-1}(p)$ using $(\succ_p)$. Then there exists a total ordering $(\succ_p)$ such that there is no bad gradient path of type $p$ in $G_I^A$.
\end{proposition}

\begin{proof}
    Suppose for the sake of contradiction that there exists a bad gradient path of type $p$, no matter what $(\succ_p)$ is. Assume that this gradient path is of the form
    \[
    \sigma_1 \to \sigma_2 \to \cdots \to \sigma_l
    \]
    where $\lcm(\sigma_1)=\lcm(\sigma_2)= \cdots = \lcm(\sigma_l) = p$. By Lemma~\ref{lem:lcm-structure}, there exists a set $M\subseteq \mingens(I)\setminus \mingens(J)$ such that for any $\sigma\in \lcm^{-1}(p)$, we have
    \[
    M\subseteq \sigma\subseteq M \cup M_p.
    \]
    Thus without loss of generality, we can assume that $\{a,b,c,d\}\coloneqq M_p=\mingens(J)$ and $a\succ_p b\succ_p c\succ_p d$. 
    By Lemma~\ref{lem:gradient-path-4-monomials}, we have the following:
    \begin{enumerate}
        \item $a$ is a bridge of $\sigma_1$ and $\sigma_2=\sigma_1\setminus \{a\}$;
        \item $b=\sbridge(\sigma_1)$;
        \item $c$ is a gap, but not a true gap of $\sigma_1$;
        \item $d$ is a non-true-gap witness of $c$ in $\sigma_1$.
    \end{enumerate}
    In particular, this means $\sigma_1=M \cup \{a,b,d\}$ and $\sigma_2=M\cup \{b,d\}$. Also, $a, c$ divides $p=\lcm(\sigma_2) = \lcm(M\cup \{b,d\})$.

    Recall that the total ordering $(\succ_p)$ is, as of right now, not a fixed ordering. It means that for if $\mingens(J)=\{m_1,m_2,m_3,m_4\}$ where $m_1\succ_p m_2\succ_p m_3\succ_p m_4$, then there exists a bad gradient path such that the first element of the path is $M\cup \{m_1,m_2,m_4\}$, and is critical by definition. We will derive a contradiction by constructing a specific $(\succ)$ on $\mingens(I)$ where $m_1\succ m_2\succ m_3\succ m_4$ such that $M\cup \{m_1,m_2,m_4\}$ is not critical.

    Consider a total ordering $(\succ)$ on $\mingens(I)$ where $b\succ d\succ a\succ c$. Then since $c$ divides $p=\lcm(M\cup \{b,d,c\})$, we have $\sbridge(M\cup \{b,d,c\})=c$. By definition, $M\cup \{b,d,c\}$ is type-2, since $c$ is the smallest option for a bridge. In particular, $M\cup \{b,d,c\}$ is not critical, a contradiction, as desired.
\end{proof}

\begin{proposition}\label{prop:5}
    Given a monomial $p\in \lcm(I)$. Assume that $|M_p|=5$. Let $A$ be the homogeneous acyclic matching resulted from applying Algorithm~\ref{algorithm1} to $\lcm^{-1}(p)$ using $(\succ_p)$. Then there exists a total ordering $(\succ_p)$ such that there is no bad gradient path of type $p$ in $G_I^A$.
\end{proposition}
\begin{proof}
    Suppose for the sake of contradiction that there exists a bad gradient path of type $p$, no matter what $(\succ_p)$ is. Assume that this gradient path is of the form
    \[
    \sigma_1 \to \sigma_2 \to \cdots \to \sigma_l
    \]
    where $\lcm(\sigma_1)=\lcm(\sigma_2)= \cdots = \lcm(\sigma_l) = p$. By Lemma~\ref{lem:lcm-structure}, there exists a set $M\subseteq \mingens(I)\setminus \mingens(J)$ such that for any $\sigma\in \lcm^{-1}(p)$, we have
    \[
    M\subseteq \sigma\subseteq M \cup M_p.
    \]
    Since $J$ has at most five generators, we have $M_p=\mingens(J)$. By Lemma~\ref{lem:gradient-path-4-monomials}, there exist 4 monomials $m_1,m_2,m_3,m_4\in \mingens(J)$ where $m_1\succ_p m_2\succ_p m_3\succ_p m_4$ such that
    \begin{enumerate}
        \item $m_1$ is a bridge of $\sigma_1$ and $\sigma_2=\sigma_1\setminus \{m_1\}$;
        \item $m_2=\sbridge(\sigma_1)$;
        \item $m_3$ is a gap, but not a true gap of $\sigma_1$;
        \item $m_4$ is a non-true-gap witness of $m_3$ in $\sigma_1$.
    \end{enumerate}
    In particular, this means $\sigma_1\supseteq M\cup \{m_1,m_2,m_4\}$. Set $\{m_5\}=\mingens(J)\setminus \{m_1,m_2,m_3,m_4\}$. Then $\sigma_2$ is either $M\cup \{m_2,m_4\}$ or $M\cup \{m_2,m_4,m_5\}$. Either way, we have $p=\lcm(M\cup \{m_2,m_4,m_5\})$.

    Set $s_p\coloneqq \min\{ |\sigma| \colon \sigma\subseteq \mingens(J) \text{ such that } \lcm(M\cup \sigma) = p \}$. Roughly speaking, $s_p$ measures the smallest amount of monomials in $\mingens(J)$ so that together with $M$, their lcm is $p$. From the above observation, we have $0\leq s_p\leq 3$. We will work on the cases separately. 

    \textbf{Case 1:}  Suppose $s_p=0$. Then by Lemma~\ref{lem:losing-the-end}, any $M\subseteq \sigma \subseteq M\cup\mingens(J)$ is not critical, and thus in particular, no bad gradient path exists, a contradiction.

    \textbf{Case 2:} Suppose $s_p=1$, and let $a\in \mingens(J)$ such that $\lcm(M\cup \{a\})=p$. Set $\mingens(J)=\{a,b,c,d,e\}$ and further assume that $a\succ_p b\succ_p c\succ_p d\succ_p e$. Then by Lemma~\ref{lem:losing-the-end} again, any $M\cup \{a\}\subseteq \sigma \subseteq M\cup\mingens(J)$ is not critical. In particular, this means $a\notin \sigma_1$, and thus $\sigma_1=M\cup \{b,c,e\}$ since $b,c,d,e$ must play the role of $m_1,m_2,m_3,m_4$, respectively, in conditions (1)--(4). To derive a contradiction, we need more assumptions on our bad gradient path. By Lemma~\ref{lem:gradient-path-shape}, we assume that our bad gradient path is of the form
    \[
    \sigma_1 \to \tau_1 \to \sigma_2 \to \tau_2 \to \cdots \to \sigma_l\to \tau_l
    \]
    for some $l\geq 1$, where $(\sigma_{i+1}\to \tau_i)$ is an edge in $A$ for any $i$, and $(\sigma_i\to \tau_i)$ is an edge of $G_I$ but not in $A$. Set $\{f_i\}=\sigma_i\setminus \tau_i$ for any $i\in [l]$, and $\{g_j\}=\sigma_{j+1}\setminus \tau_j$ for any $j\in [l-1]$. Then $g_j=\sbridge(\sigma_{j+1})$ for any $j\in [l-1]$. In particular, $g_j\in \mingens(J)$ by Lemma~\ref{lem:lcm-structure}. 
    
    We claim that $a$ is not in any element of the bad gradient path. First, we have $a\notin \sigma_1$ since $\sigma_1$ is critical, and containing $a$ makes it non-critical. Thus $a\notin \tau_1$ as well. By induction, it suffices to assume that $\sigma_t$ and $\tau_t$ do not contain $a$ for some $t<l$ and prove that $\sigma_{t+1}$ does not contain $a$. Indeed, if $a$ is in $\sigma_{t+1}$, then $a=g_t=\sbridge(\sigma_{t+1})$. Since $a$ is the biggest monomial among $\mingens(J)$ with respect to $(\succ_p)$, this implies that $\sigma_{t+1}$ has exactly one bridge, contradicting the fact that $f_{t+1}$ and $g_t$ are two different bridges of $\sigma_{t+1}$. Thus the claim holds, and $a$ is not in any element of the bad gradient path. Hence it the existence of $a$ in this case is irrelevant to the existence of a bad gradient path, and thus we can assume that $\mingens(J)=\{b,c,d,e\}$. By Proposition~\ref{prop:4}, no bad gradient path exists, a contradiction.

    \textbf{Case 3:} Suppose that $s_p=2$, and let $a,b\in \mingens(J)$ such that $\lcm(M\cup \{a,b\})=p$. Set $\mingens(J)=\{a,b,c,d,e\}$.
    
    Consider a total ordering $(\succ_p)$ on $\mingens(I)$ such that $a\succ_p b\succ_p c\succ_p d\succ_p e$. Then by Lemma~\ref{lem:losing-the-end} again, any $M\cup \{a,b\}\subseteq \sigma \subseteq M\cup\mingens(J)$ is not critical. 
    
    Set $\sigma_1=M\cup V$ where $V\subseteq \mingens(J)$. By the existence of the monomials in $V$ with conditions (1)--(4), we have $3\leq |V|\leq 4$. Suppose that $|V|=4$. Then since $V$ does not contain $\{a,b\}$ as a subset, there are only two options: either $\{a,c,d,e\}$ or $\{b,c,d,e\}$. By conditions (1)--(3), $V$ does not contain a monomial $m_3$, and contains two monomials $m_1,m_2$ such that $m_1,m_2\succ m_3$. Neither option satisfies this condition. 
    
    Hence now we can assume that $|V|=3$. By similar arguments, there are only two viable options for $V$: $\{a,c,e\}$ or $\{b,c,e\}$. Then we have $m_1=a$ in the former case and $m_1=b$ in the latter case. Either way, $\sigma_2=\sigma_1\setminus \{m_1\} = M\cup \{c,e\}$. This implies that $\lcm(M\cup \{c,e\})=p$. By the symmetry between $c,d,e$, we can change our total ordering to obtain that $\lcm(M\cup \{c,d\})=\lcm(M\cup \{d,e\})=p$. Now we consider a total ordering $(>_p)$ on $\mingens(I)$ such that $c>_p d>_p e>_p a>_p b$. By our hypothesis, there exists a bad gradient path of type $p$:
    \[
    \sigma_1' \to \sigma_2' \to \cdots \to \sigma_{l'}'.
    \]
    By similar arguments as above, we also have $M\subseteq \sigma_i'\subseteq M\cup \mingens(J)$ for any $i\in [l']$. We are still in the case $s_p=2$ and since $\lcm(M\cup \{c,d\})=p$, similar arguments apply and we can deduce that $\sigma_1'$ is either $M\cup \{c,e,b\}$ or $M\cup \{d,e,b\}$. Either way, due to conditions (1)--(4), we have $e=\sbridge_{>_p}(\sigma_1')$. Since we know that $b$ divides  $\lcm(M\cup \{c,e\})$ and $\lcm(M\cup \{d,e\})$, the monomial $b<_p e$ is a bridge of $\sigma_1$, a contradiction.

    \textbf{Case 4: } Suppose that $s_p=3$, and let $a,b,c\in \mingens(J)$ such that $\lcm(M\cup \{a,b,c\})=p$. Set $\mingens(J)=\{a,b,c,d,e\}$.
    
    Consider a total ordering $(\succ_p)$ on $\mingens(I)$ such that $a\succ_p b\succ_p c\succ_p d\succ_p e$. Then by Lemma~\ref{lem:losing-the-end} again, any $M\cup \{a,b\}\subseteq \sigma \subseteq M\cup\mingens(J)$ is not critical. Set $\sigma_1=M\cup V$ where $V\subseteq \mingens(J)$. By the existence of the monomials in $V$ with conditions (1)--(4), we have $3\leq |V|\leq 4$. If $|V|=3$, then $\sigma_2=M\cup V'$ where $|V'|=2$. This implies that $s_p\leq 2$, a contradiction. Thus $|V|=4$. Since $V$ does not contain $\{a,b,c\}$ as a subset, there are only three options: $\{a,b,d,e\}$, $\{a,c,d,e\}$, or $\{b,c,d,e\}$. By conditions (1)--(3), $V$ does not contain a monomial $m_3$, and contains two monomials $m_1,m_2$ such that $m_1,m_2\succ m_3$. Thus $V$ must be $\{a,b,d,e\}$. This implies that $\sigma_2=M\cup \{b,d,e\}$. Thus $\lcm(M\cup \{ b,d,e\})=p$. By the symmetry among $a,b,c$, we can use a similar total ordering on $\mingens(I)$ to deduce that $\lcm(M\cup \{ a,d,e\})=\lcm(M\cup \{ c,d,e\})=p$ as well. We note that condition (4) forces $m_4$ is not a bridge of $\sigma_1$. Since $m_4$ can be either $d$ or $e$, we must have 
    \begin{equation}\label{star-condition}
        \text{either } d\nmid \lcm(M\cup \{a,b,e\}) \text{ or } e\nmid \lcm(M\cup \{a,b,d\}).
    \end{equation}
    We will derive a contradiction by showing that this is not the case. 

    Now we consider a total ordering $(\succ_p)$ on $\mingens(I)$ such that $c\succ_p d\succ_p e\succ_p a \succ_p b$. We note that $\lcm(M\cup \{c,d,e\})=p$. Thus similar arguments as above apply, and we obtain that $\lcm(M\cup \{ d,a,b\})=\lcm(M\cup \{ e,a,b\})=p$, contradicting (\ref{star-condition}). This concludes the proof. \qedhere
\end{proof}

Now we are ready to prove the \ref{main}.

\begin{proof}[Proof of \ref{main}:]
    By Lemma~\ref{lem:bad-gradient-path-exists}, it suffices to show that there exists a family of total orderings $(\succ_p)_{p\in\lcm(I)}$ on $\mingens(I)$ such that if $A=\cup_{p\in \mathbb{Z}^N_{\geq 0}} A_{\succ_p}$ is the homogeneous acyclic matching from Theorem~\ref{thm:generalized-BM}, where $A_{\succ_p}$ is obtained by applying Algorithm~\ref{algorithm1} to $\lcm^{-1}(p)$, then $G_I^A$ has no bad gradient path. Indeed, since $\mingens(I)=\mingens(J)\cup \{x_{1}^{n_1},\dots, x_{N}^{n_N}\}$, any member of the lcm lattice of $J$ is a member of the lcm lattice of $I$. Thus, a family of total orderings  $(\succ_p)_{p\in\lcm(I)}$ on $\mingens(I)$ in particular gives a family of total ordering on $\mingens(J)$, and hence a homogeneous acyclic matching $A_J$ of $J$ that is a subset of $A$. If $G_I^A$ does not have a bad gradient path, then neither does $G_J^{A_J}$.
    
    Since the lcms of all elements of a bad gradient path are the same, it suffices to fix a monomial $p\in\lcm(I)$ and show that the subgraph of $G_I^A$ induced by the vertices $\lcm^{-1}(p)$ has no bad gradient path. Since subsets of $\mingens(I)$ whose lcm is not $p$ do not affect the existence of a bad gradient path of type $p$, it suffices to show that $G_I^{A_{\succ_p}}$ does not have a bad gradient path of type $p$.
    
    Indeed, first remark that $|M_p|\leq |\mingens(J)|=5$. If $|M_p|\leq 3$, then no gradient path exists by Lemma~\ref{lem:gradient-path-4-monomials} no matter what $(\succ_p)$ is. On the other hand, $|M_p|$ equals $4$ or $5$, then there exists a total ordering $(\succ_p)$ such that $G_I^{A_{\succ_p}}$ does not have a bad gradient path of type $p$ by Propositions~\ref{prop:4} and \ref{prop:5}, respectively. This concludes the proof.
\end{proof}

\bibliographystyle{amsplain}
\bibliography{refs}
\end{document}